\documentclass{article}

\usepackage{janeshi}
\usepackage{subfiles}

\newcommand\SmallFactor{10}
\newcommand\BigFactor{4100}

\begin{document}

\title{Lifting $L$-polynomials of genus $2$ curves}
\author{Jia Shi}
\date{\vspace{-5ex}}

\maketitle

\begin{abstract}

Let $C$ be a genus $2$ curve over $\Q$. Harvey and Sutherland's 
implementation of Harvey's average polynomial-time algorithm computes 
the  $\bmod \ p$ reduction of the numerator of the zeta function of $C$ 
at all good primes $p\leq B$ in $O(B\log^{3+o(1)}B)$ time, which is 
$O(\log^{4+o(1)} p)$ time on average per prime.
Alternatively, their algorithm can do this for a single good prime 
$p$ in $O(p^{1/2}\log^{1+o(1)}p)$ time. While Harvey's algorithm 
can also be used to compute the full zeta function, no practical implementation 
of this step currently exists.

In this article, we present an $O(\log^{2+o(1)}p)$ Las Vegas algorithm that 
takes the $\bmod \ p$ output of Harvey and Sutherland's implementation and 
outputs the full zeta function. We then benchmark our results against the fastest
algorithms currently available for computing the full zeta function of a genus~$2$ curve, 
finding substantial speedups in both the average polynomial-time
and single prime settings.

\end{abstract}
\section{Introduction}

Let $C$ be a genus $2$ curve over $\Q$, which we may assume is defined by a (not necessarily minimal) equation $y^2=f(x)$, where $f$ is a squarefree polynomial of degree 5 or 6 with integer coefficients.  The $L$-function of $C$ is defined by an Euler product
\[
L(C,s) := \prod_p L_p(p^{-s})^{-1} = \sum_{n\ge 1} a_n n^{-s}
\]
where $L_p(T)$ is an integer polynomial of degree at most 4.  The $L$-function $L(C,s)$ is the subject of many open conjectures in arithmetic geometry, including the paramodular conjecture, and generalizations of the Sato-Tate conjecture, the conjecture of Birch and Swinnerton-Dyer, and the Riemann hypothesis.

To investigate these conjectures, and in some cases prove instances of them, one must explicitly compute the polynomials $L_p(T)$ for all primes $p$ up to a suitable bound $B$ (one can typically take $B = O(\sqrt{N})$, where $N$ is the conductor of the Jacobian of $C$).

For odd primes that do not divide the discriminant of $f$ (all but finitely many), the polynomial $L_p(T)$ appears as the numerator of the zeta function of the genus 2 curve $C_p/\F_p$ obtained by reducing the coefficients of $f$ modulo $p$, and can be computed by counting points on $C_p$ over $\F_p$ and $\F_{p^2}$ (we consider more efficient methods below). The computation of $L_p(T)$ at odd primes of bad reduction is efficiently handled by an algorithm of Liu \cite{Liu_1994} that has been implemented in Pari/GP \cite{PARI2}.  The case $p=2$ is more difficult, and while there are algorithms to compute $L_2(T)$ that work in principle (see \cite{DD_2019}, for example), in practice these algorithms are difficult to implement efficiently.  If one is willing to assume that $L(C,s)$ satisfies its expected functional equation (as implied by the Hasse--Weil conjecture), it can be more practical to deduce $L_2(C,s)$ from the functional equation using the knowledge of $L_p(T)$ at sufficiently many odd primes $p$; see \cite[\S 5.2]{BSSVY_2016} for an explanation of this method.

The average polynomial-time algorithm of Harvey \cite{Harvey_2014, Harvey_2015} provides a method to compute $L_p(T)$ for all odd primes $p\le B$ of good reduction for $C$ in $O(B\log^{3+o(1)}B)$ time, which is $O(\log^{4+o(1)}p)$ time on average per prime $p$.  This is notably faster than any algorithm known for computing $L_p(T)$ directly, including the most efficient implementations \cite{AGS_2019} of Pila's generalization \cite{Pila_1990} of Schoof's algorithm \cite{Schoof_1985} in genus~2 (indeed, this is faster than Schoof's algorithm).

The first step of Harvey's algorithm computes $L_p(T) \bmod p$, 
which for the purpose of computing $L(C,s)$ is almost sufficient; 
one only needs $a_n$ for $n\le B$, and these $a_n$ are determined by $a_p$ for $p\le B$ 
and $a_{p^2}$ for $p^2\le B$.  For $p>64$ the value $L_p(T)\bmod p$ uniquely determines 
$a_p\in \Z$, and for $p<64$ one can use a point counting algorithm that takes negligible time.

The main contribution of this paper is the following result:

\begin{theorem} \label{thm:main-algorithm}
Let $C/\Q$ be a genus $2$ curve and let $p$ be an odd prime of good reduction for $C$. 
Given the reduction $C_p/\F_p$ and $L_p(T)\bmod p$, 
there exists a probabilistic algorithm that computes the $L$-polynomial $L_p(T) \in \Z[T]$ 
in $O(\log^{2+o(1)} p)$ expected time.
\end{theorem}

This algorithm is easy to implement and eliminates the
need to resort to a separate $p$-adic algorithm that computes $L_p(T)$ from
scratch.  
When combined with the state-of-the-art implementations of Harvey's algorithm for hyperelliptic curves described in \cite{HS_2016,Sutherland_2020}, 
this yields an algorithm to compute $L_p(T)$ for all odd good primes $p\le B$ that runs in $O(B\log^3 B)$ time,
and an algorithm to compute $L_p(T)$ for any particular prime $p$ that runs in $O(p^{1/2}\log^2 p\log\log p)$. 
Both complexity bounds improve the constant and logarithmic factors that arise in other methods that achieve roughly the same complexity.

While both approaches have an asymptotic runtime of $\tilde{O}(p^{1/2})$ for computing a single instance of $L_p(T)$, in practice,  our 
algorithm combined with an implementation of \cite{HS_2016,Sutherland_2020} achieves a speedup factor of $\SmallFactor$ over \cite{ABCMT_2019} (see Table \ref{Table-exp2}).
When computing $L_p(T)$ for all the odd primes of good reduction 
up to a bound $B$, where our algorithm has an
asymptotic advantage, the difference is more dramatic: we achieve a
speedup factor of more than $500$ already at $B=2^{13}$ and more than $1000$ for
$B=2^{21}$.

\section*{Acknowledgments}
I am very grateful to my advisor, Andrew V. Sutherland, for suggesting the problem and for his guidance throughout this project.

\section{Outline of the algorithm}

We first outline the structure of the algorithm. Throughout this paper, we assume that $C/\F_p$ is a genus~$2$ curve 
of the form $y^2=f(x)$ with $f\in \F_p[x]$ a squarefree polynomial of 
degree $5$ or $6$. In the average polynomial-time setting,
this will be the reduction of a genus $2$ curve over $\Q$ 
with good reduction at $p$.

Our algorithm takes $C$ and $L_p(T)\bmod p$ as inputs.
More precisely, the $L$-polynomial is of the form 
\[L_p(T)=p^2T^4+pa_1T^3+a_2T^2+a_1T+1 \t{ where } a_1,a_2\in\Z,\] 
and we already know the values of $a_1\bmod p$, $a_2\bmod p$. 

Let $J$ denote $\Jac(C)$, the Jacobian of $C$, and 
$\tilde{J}$ denote $\Jac(\tilde{C})$, the Jacobian of the quadratic twist of $C$.

Throughout our algorithm, we use the following important fact:
\[\#J(\F_p) = L_p(1) \text{ and } \#\tilde{J}(\F_p) = L_p(-1).\]
A proof of this fact 
can be found in Lemma 3 of \cite{Sutherland_2009}.

We furthermore assume that $p>64$, so that 
there is only one possibility for $a_1$ within the Weil bounds
\[-4\sqrt{p}\leq a_1\leq 4\sqrt{p}.\]
For $p>64$, the interval 
has width $(4\sqrt{p})-(-4\sqrt{p})=8\sqrt{p}<p$. Hence 
the value of $a_1\bmod p$ uniquely determines $a_1$. When $p<64$, we can use point counting to determine $a_1 = \#C(\F_p)-p-1$,
which is also the negation of the trace of Frobenius.

It remains to compute $a_2$. The algorithm to determine $a_2$ can be summarized in three steps. 
\begin{itemize}
    \item Step 1. 
    
    Proposition 4 in \cite{Kedlaya_2008}
    determines bounds for $a_2$ given $p$ and $a_1$.
    Define $b = \frac{a_1}{\sqrt{p}}$ and $\delta = \abs{b-4\floor{\frac{b}{4}} - 2}$.
    Then, the lower bound and the upper bound are given by 
    \begin{equation}
    B_{\text{low}}={\frac12(a_1^2-p\delta^2)},\text{ and } B_{\text{high}}=2p+\frac14 a_1^2
    \label{eqn:ksbounds}
    \end{equation}
    respectively; specifically, $B_{\text{low}} \le a_2 \le B_{\text{high}}$. We refer to these as the KS bounds. 
     Knowing
    $a_2\bmod p$ and the KS bounds enables us to produce 
    a finite list of candidates for the integer $a_2$. 
    
    For the purpose 
    of analysis, we denote these candidates as the \textit{initial candidates}.
    In Lemma \ref{InitialCandidatesLemma}, we will 
    prove that this list contains at most $5$ 
    elements. 
    
    By efficiently eliminating all the wrong 
    possibilities, the rest of the algorithm determines $a_2$, which in turn determines $L_p(T)$.
    
    \item Step 2. 
    
    Since $\#J(\F_p)=L_p(1)=p^2+pa_1+a_2+a_1+1$, we can use 
    the 2-rank of $J(\F_p)$ to eliminate 
    the candidates yielding incompatible 2-ranks. The 2-rank here refers to the
    rank of the 2-torsion subgroup of $J(\F_p)$.

    To do this, we define $L_{p,a_1,c}(T) = p^2T^4 + pa_1T^3 + cT^2 + a_1T + 1$,
    where $c$ represents different candidate values for $a_2$. We 
    eliminate any candidate $c$ such that $L_{p,a_1,c}(1)$
    is incompatible with the 2-rank.

    When the 2-rank is $1$, we know more information than just the parity of $\#J(\F_p)$.
    In fact, we can easily 
    determine $\#J(\F_p)\bmod 4$, allowing us 
    to eliminate all but one candidate. This will be explained in Section \ref{Section-step2}.

    For the purpose of analysis, we denote the candidates left after 
    completing this step 
    as the \textit{refined candidates}. If
    there is exactly one element in \textit{refined candidates}, then 
    we do not need to perform Step~3.

    \item Step 3.

    In this step, we generate a random point $P\in J(\F_p)$. Then, for each of the remaining candidates $c$ for $a_2$, we compute 
     $L_{p,a_1,c}(1)$ and test whether $L_{p,a_1,c}(1)\cdot P=0$
    in $J(\F_p) $, where $n \cdot P$ denotes the group computation
     $\underbrace{P+\ldots+P}_{n\t{ times}}$
    in $J(\F_p) $ and $0$ denotes the identity element. If $L_{p,a_1,c}(1)\cdot P\neq 0$,
     then we eliminate the candidate.

    We continue by alternating the above procedure 
    between $J(\F_p)$ and $\tilde{J}(\F_p)$, using  $L_{p,a_1,c}(1)$ 
    and $L_{p,a_1,c}(-1)$. We terminate this process 
    when only one candidate remains.

    We will prove that this process is expected to terminate within $4$ rounds
    in Section \ref{Section-step3}.

\end{itemize}

The rest of this paper is structured as follows: We quickly 
prove a result about Step 1 in Section \ref{Section-step1}.
We will describe Step 2 in detail in Section \ref{Section-step2}
and provide a proof of correctness. We then 
describe Step 3 in detail in Section \ref{Section-step3}
and give a proof of correctness. In Section \ref{Section-pseudocode}, 
we provide the pseudocode and the complexity 
analysis for the entire algorithm. We present 
benchmark results and timings in Section \ref{Section-implementation}.

\section{Initial candidates using KS bounds}\label{Section-step1}

Recall that \textit{initial candidates} denote the candidates of $a_2$
that fall within the KS bounds (see Equation \ref{eqn:ksbounds}). Let $B_{\text{low}}, B_{\text{high}}$
denote the lower KS bound and the upper KS bound, respectively. 
 We next show that there are at most $5$ elements 
in \textit{initial candidates}.

\begin{lemma} \label{InitialCandidatesLemma}
    Considering Equation \ref{eqn:ksbounds}, only one of the two following cases can happen:
    \begin{enumerate}
        \item There are at most $4$ elements in  \textit{initial candidates}  $=\set{c\in \Z: B_{\text{low}}\leq c\leq B_{\text{high}}, c\equiv  a_2\bmod p}$
        \item There are exactly $5$ elements in  \textit{initial candidates}, which are equal to $\set{-2p,-p,0,p,2p}$. 
        In this case, $a_1=0$.
    \end{enumerate}
\end{lemma}
\begin{proof}
    We calculate the width of the interval in which $a_2$ lies:
    \begin{align*}
        B_{\text{high}}-B_{\text{low}} &= 2p+\frac14 a_1^2 - \frac12(a_1^2-p\delta^2)\\
        &=2p-\frac14 a_1^2 + \frac12 p\delta^2\\
        &\leq 2p-\frac14 a_1^2 + 2p\leq 4p.
    \end{align*}

    Note that equality occurs only if $a_1=0$, which yields $B_{\text{high}}=2p$ and $B_{\text{low}}= -2p$. In this case, 
    if $a_2 \equiv 0\bmod p$ then we obtain $\set{-2p,-p,0,p,2p}$. In all the other cases,
    there are at most $4$ integers in the interval $[B_{\text{low}},B_{\text{high}}]$ that are equivalent to $a_2\bmod p$.
\end{proof}

\section{Eliminate wrong candidates using 2-rank} \label{Section-step2}
 
We define the \textit{factorization pattern} of a separable 
polynomial $g(x)\in \F_p[x]$ to be the multiset of 
the degrees of its irreducible factors.

To compute the factorization pattern, we do not need to completely
factor $g(x)$. We can compute the factorization pattern in $O(\log^2 p\log\log p)$ time using 
the distinct-degree factorization 
algorithm outlined in Chapter 14 of \cite{von_zur_Gathen_Gerhard_2013}.

We can determine the 2-rank of $J(\F_p) $ by identifying the 
factorization pattern of $f(x)$ in Table \ref{2rank-chart}. For a proof, see Lemma 
4.3 and Lemma 5.6 in \cite{MichaelStoll2001}.

\begin{table}
\begin{center}
    \begin{tabular}{ |c|c| } 
    \hline
     Factorization Pattern & 2-rank\\
    \hline
    $(5)$& $0$\\
    $(2,3)$& $1$ \\
    $(4,1)$ & $1$ \\
    $(3,1,1)$ &$2$ \\
    $(2,2,1)$ & $2$\\
    $(2,1,1,1)$ & $3$\\
    $(1,1,1,1,1)$ & $4$\\
    \hline 
    $(6)$ & $0$\\
    $(4,2)$ & $1$\\
    $(3,3)$ & $0$ \\
    $(2,2,2)$ & $2$\\
    \hline
    \end{tabular}{}
\end{center}
\caption{2-rank based on factorization pattern}\label{2rank-chart}
\end{table}

Note that we omitted factorization patterns corresponding to a degree 6 polynomial with at least one linear factor,
as such polynomials always correspond to a degree 5 model. In our algorithm, we always convert such 
cases to a degree 5 model to leverage specific optimizations.

\subsection{The algorithm} \label{subsection-2rank-algo}

Using Table \ref{2rank-chart}, we compute the 2-rank of the Jacobian and 
eliminate wrong candidates as follows: 
\begin{itemize}
    \item Case 1: If the 2-rank of $J(\F_p)$ is $0$, we only keep the candidates 
    that yield an odd Jacobian size.
    \item Case 2: If the 2-rank of  $J(\F_p)$ is $1$, 
    we only keep the candidates yielding an even Jacobian size.
    Applying Proposition \ref{prop:rank1} 
    allows us to determine whether the Jacobian size is $0\bmod 4$ or $2\bmod 4$
    and only keep the corresponding candidates.
    \item Case 3: If $r$, the 2-rank of  $J(\F_p)$, is at least $2$,
    we only keep the candidates that yield a Jacobian size divisible by $2^r$.
\end{itemize}
Recall that the candidates remaining after this elimination
process are called \textit{refined candidates}.
Combining the three above cases with Lemma \ref{InitialCandidatesLemma},
we observe that there are at most $2$ elements left in \textit{refined candidates}.
Of course, if there is only $1$ element left, we return immediately.

\subsection{Optimization when the 2-rank equals $1$}

\begin{proposition} \label{prop:rank1}
    Let $p>2$ be a prime.
    Let $C\colon y^2=f(x)$
    be a genus $2$ curve over $\F_p$ where $f$ is squarefree with 
    the 2-rank of $\Jac(C)(\F_p)$ equal to $1$. We also assume 
    that if $\deg(f)=6$, it does not have any linear factors. This can 
    be achieved by a change of variables. 
    Therefore, we can always write $f = ug_1g_2$, where $u\in\F_p^\times$, $g_1$ and $g_2$ 
    are both monic and irreducible. Then 
    \[\#J(\F_p) \equiv\begin{cases}
         0 \bmod 4 &\t{ if } \Res(g_1,g_2)\in (\F_p^{\times})^2\\
         2 \bmod 4 &\t{ otherwise. }
   \end{cases}\]
\end{proposition}
\begin{remark}
   Note that when the $2$-rank is $1$, the factorization pattern is $(1,4), (2,3)$ or $(2,4)$, with
   \[\Res(g_1,g_2)=(-1)^{\deg(g_1)\deg(g_2)}\Res(g_2,g_1)=\Res(g_2,g_1).\] 
   Hence the order in which we assign $g_1$ and $g_2$ is irrelevant.
\end{remark}

\begin{proof}
    We can assume that $f$ is monic after a change of variables
    and possibly replacing $C$ by its quadratic twist. This is allowed since
    \[4\mid 2a_1(p+1) = \#J(\F_p) - \#\tilde{J}(\F_p),\]
    so the residue modulo $4$  is the same 
    whether we are working with $C$ or its twist.

    Since the 2-rank is $1$, we know that the 2-Sylow subgroup of $J(\F_p)$ is isomorphic 
    to  $\Z/2^q\Z$ for some $q\geq 1$.

    Let $P$ be the generator of $J(\F_p)[2]$, or equivalently the unique nonzero element 
    such that  $2\cdot P=0$. $P$'s image in the 2-Sylow subgroup corresponds 
    to the element $2^{q-1}$ in  $\Z/2^q\Z$. If the 2-Sylow subgroup is isomorphic to 
    $\Z/2\Z$, then there does not exist an element $Q\in J(\F_p)$ such that 
    $2Q=P$. On the other hand, if $q\geq2$, there always exists an 
    element  $Q\in J(\F_p)$ where $2Q=P$. The former case corresponds to 
    the situation where $\# J(\F_p)\equiv 2 \bmod 4$, and the latter corresponds 
    to the situation where $\# J(\F_p)\equiv 0 \bmod 4$. Hence $\# J(\F_p)\equiv 0 \bmod 4$ if and only if $P\in 2J(\F_p)$.

    Let $A=\F_p[x]/f(x)$ be the étale algebra  defined by $f(x)$. In the case when 2-rank is $1$, by Table \ref{2rank-chart}, 
     $f(x)$ is the product of two irreducible polynomials in $\F_p[x]$. 
    
    We write $f(x)=g_1(x)g_2(x)$, where $g_1(x), g_2(x)$ are both monic and irreducible. $A$ can be written as a product of two finite field extensions of $\F_p$:
    \[A\simeq \underbrace{\frac{\F_p[x]}{g_1(x)}}_{:=A_1}\times  \underbrace{\frac{\F_p[x]}{g_2(x)}}_{:=A_2}. \]

    In \cite{MichaelStoll2001}, Stoll  presented an injective homomorphism 
    \[\delta: J(\F_p)/2J(\F_p)\to A^\times/A^{\times2},\]
    along with an explicit method to compute the image of the generator.
    If $P$ is the generator of $J(\F_p)[2]$, and $\theta$ is the image of $x$ in $A^\times$, then the image is given by

    \begin{equation}\label{eq:delta-formula}\delta(P) = \begin{cases}
        (-1)^{\deg(g_1)}g_1(\theta) + (-1)^{\deg(g_2)} g_2(\theta) \bmod A^{\times2} & \t{if }f \t{ has odd degree} \\
        g_1(\theta) - g_2(\theta)  \bmod A^{\times2} & \t{if }f \t{ has even degree.}
    \end{cases}
    \end{equation}

    Therefore, $P\in 2J(\F_p)$ if and only if $P\in\ker(\delta)$, which occurs 
    if and only if $\delta(P)\in A^{\times2}$. Hence this amounts to computing whether $\delta(P)$
    is a square in $A^\times$. This is equivalent to identifying whether 
    its image in $A_1$ and $A_2$ are both squares. We further use the fact that in finite fields, 
    an element is a square if and only if its norm is.
    
    Without loss of generality, in the case 
    where the factorization pattern is $(1,4)$ or $(2,3)$, we can assign $g_1$
    to be the even-degree factor and $g_2$ to be the odd-degree factor. 
    When the factorization pattern is $(2,4)$, we assign $g_1$ to be 
    the degree $2$ factor and $g_2$ to be the degree $4$ factor.

    By \eqref{eq:delta-formula}, 
        \[\delta(P)=g_1(\theta)-g_2(\theta)\in (A_1\times A_2)/(A_1^2\times A_2^2).\]
        Let $\alpha_1,\alpha_2$ be the image of $x$ in $A_1$ and $A_2$, respectively. The 
        above is equal to 
        \[(-g_2(\alpha_1), g_1(\alpha_2))\in (A_1\times A_2)/(A_1^2\times A_2^2).\]
        Computing the norm yields:
        \begin{align*}
            N_{A_1/\F_p} (g_2(\alpha_1)) &= \prod_{\sigma\in \Gal(A_1/\F_p)}\sigma(g_2(\alpha_1))\\
            &= \prod_{\sigma\in \Gal(A_1/\F_p)}g_2(\sigma(\alpha_1))\\
            &=\prod_{\beta\in \t{roots}(g_1)}g_2(\beta)\\
            &=\Res(g_1,g_2).
        \end{align*}

       Since $A_1\simeq \F_p[x]/g_1(x)$ and $g_1$ has even degree,
        $A_1\simeq \F_{p^n}$ where $n$ is even. So 
        \[N(-1)=(-1)^{\frac{p^{n}-1}{p-1}}=(-1)^{p^{n-1}+p^{n-2}+\ldots+p+1} = 1.\]
        Hence $-g_2(\alpha_1)$ is a square in $A_1$ $\iff$ $g_2(\alpha_1)$ is a square in $A_1$ $\iff$ $\Res(g_1,g_2)$ is a square in $\F_p$. 
        
        On the other hand, 
        \[N_{A_2/\F_p} (g_1(\alpha_2))= \Res(g_2,g_1)=(-1)^{(\deg(g_1))(\deg(g_2))}\Res(g_1,g_2)=\Res(g_1,g_2).\]

        Therefore, it suffices to check whether $\Res(g_1,g_2)$ is a square in $\F_p$.
\end{proof}

\begin{remark}
    In \cite{MichaelStoll2001}, although Lemma 4.3 and Lemma 5.6 are specific for 
    polynomials defined over $\Q$, the same arguments also work in $\F_p$.
\end{remark}

\begin{corollary}
    When the 2-rank of $J(\F_p)$ equals $1$, knowing $\#J(\F_p)\bmod 4$ 
    allows us to eliminate all but $1$ \textit{initial candidates}.
\end{corollary}

\begin{proof}
    If there are $4$ or fewer \textit{initial candidates}, then
    \textit{initial candidates} come from an arithmetic sequence
    with difference $p$ and length at most $4$. Knowing $\#J(\F_p)\bmod 4$ leaves us with only 
    one candidate.

    If there are $5$ \textit{initial candidates}, then the only 
    case when knowing $a_2\bmod 2p$ does not yield 
    a unique solution is when  $a_2\in\set{-2p,2p}$, see Lemma \ref{InitialCandidatesLemma}.
    We will show that this cannot happen.

    Combined with the fact that $a_1=0$, we have $L_p(T)=(pT^2+1)^2$ or $L_p(T)=(pT^2-1)^2$. The characteristic polynomial 
    of the Frobenius endomorphism is either $(T^2+p)^2$ or $(T^2-p)^2$. By Theorem 1.1 of \cite{ZHU2000292},
    $J(\F_p)$ is an elementary supersingular abelian variety whose group structure is  
    isomorphic to one of the following: 
    \[(\Z/g(1)\Z)^2, \  \Z/g(1)\Z\times\Z/\tfrac{g(1)}{2}\Z\times\Z/2\Z, \t{ or }(\Z/\tfrac{g(1)}{2}\Z\times\Z/2\Z)^2\]
    where $g(T)=(T^2+p)$ or $g(T)=(T^2-p)$.

    In all three cases, the 2-rank of $J(\F_p)$ is at least $2$. This 
    contradicts the hypothesis that the 2-rank is $1$.
\end{proof}
\begin{remark}

    One warning is that there are cases when knowing $\#J(\F_p)\bmod 4$ is not enough.
    When the 2-rank is $2$, and the two \textit{refined candidates} are $\set{-2p,2p}$,
    Jacobian arithmetic is still needed to rule out the wrong candidate. Our experiments 
    have found such cases.
\end{remark}

\begin{remark}
    We ran an experiment using the lifting algorithm (Algorithm \ref{alg:lift})
    on 4000 random curve-prime pairs with $p \le 2^{16}$. Roughly $19\%$ of the curves 
    required Step 3, while the rest terminated after Step 2.

    The observed 2-rank distribution was as follows: 42.02\% are of rank 0, 41.50\% are of rank 1, 14.57\% are of rank 2, 1.70\% are of rank 3, and 0.20\% are of rank 4.
    Out of the instances where the 2-rank is $1$, about $48\%$ of them 
    had more than one candidate yielding even Jacobian sizes, and 
    applying the optimization allows us to determine the correct one using the Jacobian size modulo $4$. Hence 
    we no longer need Step 3 for these instances, which is expensive due to the Jacobian arithmetic.

    In our experiment, Step 2 speeds up 
    the lifting process by approximately $2$--$4$ times overall.
\end{remark}

\section{Eliminate wrong candidates using Jacobian arithmetic} \label{Section-step3}

\subsection{The algorithm}\label{subsection-back-forth}

In this step, we generate a random point $P$ in $J(\F_p)$ and, 
for each candidate $c$ for $a_2$, check whether 
 $L_{p,a_1,c} (1) \cdot  P = 0$.
We only keep the ones that yield the identity. We alternate the above process between the Jacobian and its twist.
We use a randomized process in this step, 
although all the prior steps are deterministic.

\begin{example}
    Note that running the elimination process 
    on the twist is necessary in some cases:
    we found examples where $J(\F_p)\simeq \Z/p\Z\times\Z/p\Z$, $a_1=-2, a_2=2p+1$,
    and the wrong candidate equals $1$. In these cases, the wrong Jacobian size would be $p(p-2)$.
    However, for any point $P\in J(\F_p)$, we would have $p(p-2)\cdot P=0$.
    Hence it is not possible to distinguish the wrong candidate 
    from the right one by performing Jacobian operations on just $J(\F_p)$. We 
    also need to use $\tilde{J}(\F_p)$, but this is the only situation where 
    using $\tilde{J}(\F_p)$ is necessary 
    (see Proposition \ref{non-terminate-condition}).
    
    One example is the following:
    \[C\colon y^2 = x^6 + 202 x^5 + 77 x^4 + 152 x^3 + 153 x^2 + 34 x + 283, \ \ p=313.\]
   
    We run our algorithm on this input. \textit{RefinedCandidates}
    consists of $1$ and $627$, 
    yielding
    \[\begin{cases}
        L_{p,a_1,1}(1)&= 97343=313\cdot311=p\cdot(p-2)\\
        L_{p,a_1,627}(1)&= 97969=313^2=p^2\\
    \end{cases} \text{ and }
    \begin{cases}
        L_{p,a_1,1}(-1)&= 98599=43\cdot2293\\
        L_{p,a_1,627}(-1)&=99225= 3^4\cdot 5^2\cdot7^2\\
    \end{cases}
    \]  
    Since the group exponent of $J(\F_p)$, $313$, divides both 
    $L_{p,a_1,1}(1)$ and $L_{p,a_1,627}(1)$, regardless of which 
    random point $P\in J(\F_p)$ we pick, it will be annihilated 
    by both of them. Hence it is impossible to identify the correct candidate from 
    the Jacobian arithmetic on $J(\F_p)$ alone.
    We then pick a random point $P'\in \tilde{J}(\F_p)$. 
    Then $L_{p,a_1,1}(-1)$ does not annihilate
    $P'$, but $L_{p,a_1,627}(-1)$ does. Therefore, the correct answer is $627$.
    In Proposition \ref{non-terminate-condition} and Proposition \ref{prop-twist},
    we will show that whenever we cannot distinguish the correct candidate from
    Jacobian arithmetic in  $J(\F_p)$, we will always be able to do so in $\tilde{J}(\F_p)$.

\end{example}

\subsection{Proof of correctness and expected termination}
We will prove that this step is expected to terminate and return the correct $L$-polynomial.

The following proposition shows that,
unless $J(\F_p)\simeq \Z/p\Z\times\Z/p\Z$ and the wrong 
candidate for $a_2$ yields a Jacobian size of $p^2+2p$ or $p^2-2p$, there always 
exists a point $P\in J(\F_p) $ to 
guarantee the termination of the procedure described in Subsection \ref{subsection-back-forth}.

\begin{proposition} \label{non-terminate-condition}

    We assume $p>47$.

    If there are two candidates $c,c'$, in \textit{refined candidates}, assuming $c < c'$,
    such that all points $Q\in J(\F_p)$ yield  $L_{p,a_1,c}(1)\cdot Q = L_{p,a_1,c'}(1) \cdot Q= 0$,
    then it must be the case that 
    \[a_2=c,  \t{ and }\begin{cases}
        L_{p,a_1,c}(1)&=p^2\\
        L_{p,a_1,c'}(1)&=p^2+2p\\
    \end{cases}\]
    or 
    \[a_2=c',  \t{ and }\begin{cases}
        L_{p,a_1,c}(1)&=p^2-2p\\
        L_{p,a_1,c'}(1)&=p^2.\\
    \end{cases}\]
\end{proposition}

In other words, unless we are in the case described above, there always exists a point $P\in J(\F_p)$
that is not annihilated by the Jacobian size coming from a wrong candidate, allowing 
us to eliminate it. In 
Proposition~\ref{prop-expect-terminate}, we will show that unless we're in the above case, 
the probability of drawing such a point is at least $1/2$.

\begin{proof}

    The 2-rank information reveals $L_{p,a_1,a_2}(1)\bmod 2$ (though when the 2-rank is not $0$ or $1$,
    even more information is given), which implies we know $a_2$ mod $2p$.
    Using Lemma \ref{InitialCandidatesLemma}, either 
    \begin{enumerate}
        \item $c'=c+2p$ or 
        \item $c'=c+4p$ and there were exactly $5$ elements in  \textit{initial candidates}.
    \end{enumerate}

    We prove these cases separately.

    Write $J(\F_p) \simeq \Z/n_1\Z\times \Z/n_2\Z\times\Z/n_3\Z\times\Z/n_4\Z$
    where $n_1\mid n_2\mid n_3\mid n_4$. This is always possible, 
    since for every prime power $l^n$, the $l^n$-torsion subgroup 
    of an abelian surface has rank at most $4$.

    \begin{enumerate}
        \item  

        Suppose $c'=c+2p$.
        Since for all $Q\in J(\F_p)$, we have \[L_{p,a_1,c}(1) \cdot Q = (p^2+pa_1+c+a_1+1)\cdot Q=0\]
        and \[L_{p,a_1,c'}(1) \cdot Q= (p^2+pa_1+(c+2p)+a_1+1)\cdot Q=0\]
        we deduce that $n_4\mid 2p$. When $p\geq 23$, we compute that $L_{p,a_1,a_2}(1)>16$.
        Hence $n_4>2$ and we must have $n_4=p$ or $2p$. 

        \begin{enumerate}
            \item Suppose $n_4=p$. Then, $L_{p,a_1,a_2}(1)$ must be a power 
             of $p$. For $p\geq 23$, the lower Weil bound strictly exceeds $p$,
            and for $p\geq 5$, the upper Weil bound is strictly smaller than $p^3$.
            
            Therefore, we must have $L_{p,a_1,a_2}(1)=p^2$, and $n_1=n_2=1, n_3=n_4=p$. 
            
            We must land in one of the following cases:
            
            \[a_2=c,  \t{ and }\begin{cases}
                L_{p,a_1,c}(1)&=p^2\\
                L_{p,a_1,c'}(1)&=p^2+2p\\
            \end{cases}\]
            or 
            \[a_2=c',  \t{ and }\begin{cases}
                L_{p,a_1,c}(1)&=p^2-2p\\
                L_{p,a_1,c'}(1)&=p^2\\
            \end{cases}\]

            \item Suppose $n_4=2p$. 
            
            We first show that $p\nmid n_3$.
            If not, $L_{p,a_1,a_2}(1)\geq 2p^2$. However,
            whenever $p\geq29$, the upper Weil bound is always 
            less than $2p^2$, yielding a contradiction. Therefore, $p\nmid n_3$.\\
            This means each of $n_1,n_2,n_3$
            equals either $1$ or $2$. This yields $n_1n_2n_3n_4=L_{p,a_1,a_2}(1)\in\set{2p,4p,8p,16p}$, 
            but for $p\geq 47$, the lower Weil bound strictly exceeds $16p$.
            This yields a contradiction.
        \end{enumerate}

        \item 

        In this case, $c'=c+4p$  and $a_1=0$.
        From \[L_{p,a_1,c}(1)\cdot  Q = (p^2+pa_1+c+a_1+1)Q=(p^2+c+1)\cdot Q= 0\]
        and \[L_{p,a_1,c'}(1)\cdot Q= (p^2+c+4p+1)Q=0,\] we deduce that $(4p)\cdot Q = 0$ 
        for all $Q\in J(\F_p)$. Therefore $n_4\mid 4p$. When $p\geq37$, $L_{p,a_1,a_2}(1)>256$, 
        thus $n_4>4$. Hence, $n_4=p, 2p$ or $4p$. In any case,
        $p^2 + c+1 \equiv 0\bmod p$ and  $p^2 + c'+1 \equiv 0\bmod p$,
        hence $c\equiv c'\equiv -1\bmod p$, but this is a contradiction,
        as $c=-2p$ and $c'=2p$.
    \end{enumerate}
\end{proof}

\begin{proposition}\label{prop-expect-terminate}
    Suppose $J(\F_p)\not\simeq \Z/p\Z\times\Z/p\Z$, and 
    assume without loss of generality that $c$ denotes the wrong candidate and $c'$ denotes the correct candidate.

    If we draw a uniformly random point $P_1$ from $J(\F_p)$ and then draw a second 
    uniformly random point $P_2$, the probability that at least one of these two points 
    eliminates the wrong candidate (i.e., $L_{p,a_1,c}(1)\cdot P_1\neq 0$ or $L_{p,a_1,c}(1)\cdot P_2\neq 0$) is at least $1/2$.
    This implies that the expected number of iterations on $J(\F_p)$ before termination is at most $4$.
\end{proposition}

\begin{proof}
    Write $J(\F_p) \simeq \Z/n_1\Z\times \Z/n_2\Z\times\Z/n_3\Z\times\Z/n_4\Z$
    where $n_1\mid n_2\mid n_3\mid n_4$; then $n_4$ is the group exponent of $J(\F_p)$.
    We assume each point of $J(\F_p)$ is equally likely to be drawn. Let $P_1$ and $P_2$
    be two randomly drawn points.
    By Theorem 8.1 in \cite{Sut07}, the probability of $\lcm(\abs{P_1}, \abs{P_2})=n_4$ is at least $\frac{6}{\pi^2}>\frac12$.

    Suppose we are in the case where $\lcm(\abs{P_1}, \abs{P_2})=n_4$.
    We will show that we must be able to eliminate the wrong candidate $c$. To do this,
    suppose $c$ is not eliminated. This means that both $P_1$ and $P_2$ are annihilated by $L_{p,a_1,c}(1)$
    and $L_{p,a_1,c'}(1)$. Hence, we must have $\abs{P_i}\mid L_{p,a_1,c}(1)$ and $\abs{P_i}\mid L_{p,a_1,c'}(1)$ for $i=1,2$.
    Thus, both the candidate Jacobian sizes are divisible by $\lcm(\abs{P_1}, \abs{P_2}) = n_4$
    and will annihilate all points on $J(\F_p)$.
    
    Following the proof of 
    Proposition \ref{non-terminate-condition}, we must have $J(\F_p)\simeq \Z/p\Z\times\Z/p\Z$. This 
    contradicts our assumption. Therefore, with probability at least $\frac12$, we can eliminate 
    the wrong candidate.

    The expected number of points we need to test is at most $2\cdot \frac{1}{1-\frac12}=4$.
\end{proof}

\begin{remark}
    We note that $4$ is a very conservative upper bound on the expected number 
    of iterations. In practice, we rarely need more than one round.
\end{remark}

When $J(\F_p)\simeq \Z/p\Z\times\Z/p\Z$, it is possible 
that the elimination process using random points on $J(\F_p)$ does not terminate. However, 
Proposition \ref{prop-twist} shows that the twist must not be isomorphic to $\Z/p\Z \times \Z/p\Z$.

As a result, using Proposition \ref{non-terminate-condition} and Proposition \ref{prop-expect-terminate}, the 
process on $\tilde{J}(\F_p)$ is expected to terminate after at most $4$ iterations. In conclusion,
the algorithm that alternates between $J(\F_p)$ and $\tilde{J}(\F_p)$ is expected to terminate after 
at most $8$ iterations.

\begin{proposition}\label{prop-twist}
    It is not possible that $J(\F_p)\simeq \Z/p\Z \times \Z/p\Z $ and $\tilde{J}(\F_p)\simeq \Z/p\Z \times \Z/p\Z$.
\end{proposition}

\begin{proof}
    Suppose towards contradiction that both are isomorphic to $\Z/p\Z \times \Z/p\Z$.

    Let $\phi_p$ be the Frobenius endomorphism on $J(\F_{p^2})$ induced by the Frobenius 
    automorphism $\alpha\mapsto \alpha^p$ on $\F_{p}$.
    Note that $J(\F_{p^2}) = \ker(\phi_p^2-1)$ and $J(\F_p) = \ker(\phi_p-1)$.

    After identifying $J(\F_p)$ and $\tilde{J}(\F_p)$ over $\F_{p^2}$,
    the $\F_p$-points of $\tilde{J}(\F_p)$ are exactly the
    points where $\phi_p$ acts by~$-1$. Therefore, we may 
    identify $\tilde{J}(\F_p)$ with $\ker(\phi_p+1)$ as a subvariety of $J(\F_{p^2})$.
    
    Since there is no $2$-torsion point in $J(\F_p)$,
    $J(\F_p)=\ker(\phi_p-1)$ and $\tilde{J}(\F_p)=\ker(\phi_p+1)$ intersect trivially
    inside $J(\F_{p^2})$.
    Thus, the $p$-rank of $J(\F_{p^2})$
    is at least the sum of the $p$-ranks of $J(\F_p)$ and $\tilde{J}(\F_p)$, which is $4$.

    Using the fact that the $p$-torsion of 
    the Jacobian of a genus $g$ hyperelliptic curve (defined 
    over a field of characteristic $p$) 
    cannot exceed $g$ (see Page 147 in \cite{Mum08} or tag 03RP 
    in \cite{stacks-project} for a proof), we conclude that the $p$-rank of $J(\F_{p^2})$
    cannot exceed $2$.
    This yields a contradiction and completes the proof.
\end{proof}

\section{Pseudocode and complexity analysis}
\label{Section-pseudocode}

\begin{algorithm}
    \caption{Lift-l-poly}\label{alg:lift}
    \begin{algorithmic}
    \State \textbf{Input:} a genus $2$ curve $C$ given by $y^2=f(x)$ where $f\in \F_p[x]$ is a separable polynomial of degree $5$ or $6$, and $a_1',a_2'\in\F_p$,
    such that $L_p(T) \equiv a_2'T^2+a_1'T+1 \pmod p$.
    We furthermore assume $p>64$.
    \State \textbf{Output:} the $L$-polynomial $p^2T^4+pa_1T^3+a_2T^2+a_1T+1\in\Z[T]$
    \State 

    \If{$\deg(f) = 6$ and $f$ has a linear factor}
        \State Change $f$ to a degree $5$ model
    \EndIf

    \State Define $a_1$ to be the unique element in $(-p/2,p/2)$ congruent to $a_1'\bmod p$
    \State Define the KS bounds $B_{\text{low}},B_{\text{high}}$ for $a_2$, where $B_{\text{low}},B_{\text{high}}\in \R$. See Equation \ref{eqn:ksbounds}
    \State Define $\mathit{InitialCandidates}$ $:=[c\in \Z, B_{\text{low}}\leq c\leq B_{\text{high}} \mid c\equiv a_2' \bmod p]$
    \State Compute $r = \mathit{Jacobian2Rank}(f)$. See Table \ref{2rank-chart}.
    \State Define $L_{p,a_1,c}(T) = p^2T^4+pa_1T^3+cT^2+a_1T+1$, for different choices of $c$ in $\mathit{InitialCandidates}$.

    \If{$r = 0$}
        \State Define $\mathit{RefinedCandidates} =[c\in \mathit{InitialCandidates} \mid L_{p,a_1,c}(1)\equiv1\bmod 2$]
    \ElsIf{$r = 1$}
        \State Compute the variable $\mathit{residueMod4}$. See Case 2 in Section \ref{subsection-2rank-algo}.
        \State Define $\mathit{RefinedCandidates} =[c\in \mathit{InitialCandidates} \mid L_{p,a_1,c}(1)\equiv \mathit{residueMod4} \bmod 4$]
    \Else 
        \State Comment: in this case, $r\geq 2$.
        \State Define $\mathit{RefinedCandidates} =[c\in \mathit{InitialCandidates} \mid L_{p,a_1,c}(1) \equiv 0\bmod 2^r$]
    \EndIf

    \State Define $\mathit{Iterator} = 0$
    \While{$\#\mathit{RefinedCandidates}>1$}: 
        \If{$\mathit{Iterator}\equiv 0\bmod 2$}
            \State Generate random point $P\in J(\F_p)$
            \State Replace $\mathit{RefinedCandidates}$ by those $c$ that yield $L_{p,a_1,c}(1)\cdot P = 0$.
        \Else 
            \State Generate random point $P\in \tilde{J}(\F_p)$ where $\tilde{J}$ is the Jacobian of the quadratic twist of $C$
            \State Replace $\mathit{RefinedCandidates}$ by those $c$ that yield $(p^2-a_1p+c-a_1+1)\cdot P = 0$.
        \EndIf
        \State Increment $\mathit{Iterator}$ by $1$
    \EndWhile
    \State Let $c$ be the one candidate left, define $a_2=c$
    \State \Return $L_{p,a_1,a_2}(T)$. 
    \end{algorithmic}
    \end{algorithm}

This algorithm is a Las Vegas algorithm, 
which always yields correct results when it terminates, and 
has finite expected running time. We compute the worst possible expected running time.
We compute bit complexity and we use $n=\log p$ to denote the number of bits needed 
to represent an element of $\F_p$.

We first analyze the time complexity for computing the KS bounds (see Equation \ref{eqn:ksbounds}), the 2-rank,
and the residue modulo $4$ when the 2-rank equals $1$.

The time complexity for calculating the KS bounds 
is $O(\mathsf{M}(n))$, which is asymptotically equivalent to $O(n\log n)$ by \cite{harvey2021integer}. 
The time complexity for computing the 2-rank is $O(n^2\log n)$ (or equivalently, $O(\log^2 p \log\log p)$), as seen in Section~\ref{Section-step2}.
The time complexity for computing the residue modulo $4$ when the 2-rank equals $1$
is dominated by the complexity to compute 
the resultant of two polynomials of degree at most $4$.
By the product formula for resultants, it can be computed 
after a constant number of ring operations,
 so the time complexity is also $O(\mathsf{M}(n))=O(n\log n)$.

Now that we have analyzed the time complexity of the smaller components, we can analyze the time complexity of
Algorithm \ref{alg:lift}. Note that there are at most $2$ elements in \textit{RefinedCandidates},
and we only run the loop when there are exactly $2$ elements.
For each iteration of the loop, we generate a random point and perform 
Jacobian scalar multiplication for each of the two candidates in \textit{RefinedCandidates}.

We generate random points $P=(u(x),v(x))$ on the Jacobian as semi-reduced divisors 
in Mumford representation (see Chapter 10 of \cite{10.5555/2230462} for details). We first generate a random monic polynomial $u(x)\in \F_p[x]$
of  degree $1$ or $2$,
and check whether $f(x)$ is a square modulo $u(x)$. When $u(x)$ is separable, we consider the étale algebra $\F_p[x]/u(x)\simeq F_1$ or $F_1\times F_2$,
where each $F_i$ is a finite field of characteristic $p$. $f(x)$ is a square in $\F_p[x]/u(x)$
if and only if its image in each of the $F_i$ is a square. When $u(x)$ is not separable, it is
of the form $(x-a)^2$ for some $a\in \F_p$. There are $1+\frac{(p-1)p}{2}$ possible squares 
modulo $(x-a)^2$.
If $f(x)$ is a square, we set $v(x)$ as one of the square roots at random, otherwise 
we generate $u(x)$ again and repeat. Therefore, the probability that $f(x)\in \F_p[x]/u(x)$
is a square is at least $\left(\frac12\right)^2\geq\frac{1}{4}$. Hence, the expected number of 
$u(x)$ we generate until $f(x)\bmod u(x)$ is a square is constant. Furthermore, since there are $p$
linear monic polynomials and $p^2$ degree~$2$ monic polynomials in $\F_p[x]$, and there 
are at most $2^2=4$ possible square roots for $f(x)\bmod u(x)$, the probability of 
generating each non-trivial point in $J(\F_p)$ (or $\tilde{J}(\F_p))$ is at least $\frac{1}{4(p+p^2)}$, 
allowing us to sample points on the Jacobian with a nearly uniform 
distribution (which suffices for our purposes).
Therefore, the expected time complexity for generating a random point on 
$J(\F_p)$ (or $\tilde{J}(\F_p)$) is bounded above by the expected time complexity 
to compute a square root in $\F_p[x]/u(x)$, which is $O(n\mathsf{M}(n))=O(n^2\log n)$
by the Cantor--Zassenhaus algorithm. 
The time complexity for 
computing scalar multiplication on the Jacobian using double and add 
is $O(n\mathsf{M}(n))=O(n^2\log n)$. Therefore, the expected time complexity for completing one iteration of the 
loop is $O(n^2\log n)$. The expected number of iterations is constant, as addressed in 
Section~\ref{Section-step3}.

Combining the above, we conclude that the  expected time complexity to lift $L_p(T)$
from $\F_p[x]$ to $\Z[T]$ is $O(n^2\log n)$ or $O(\log^2 p\log\log p)$ per prime. 

To compute $L_p(T)\bmod p$, it takes $O(\log^{4+o(1)}p)$ time on average. Therefore, 
the total time to compute $L_p(T)$ is still $O(\log^{4+o(1)}p)$ on average and the 
 time complexity for computing the $L$-polynomial up to a bound $p\leq B$ remains
 $O(B\log^{3+o(1)}B)$ time.

The space complexity for our algorithm is $O(\log p)$ since at any 
point throughout the computation, we store a constant number of polynomials of 
degree at most $6$ with $\F_p$ coefficients.

\section{Implementation and comparisons}\label{Section-implementation}

An implementation of Algorithm \ref{alg:lift} in MAGMA (\cite{MR1484478}, version 2.28-25)
 is available on \cite{Shi_2025}. Our experiments use an implementation of the algorithm described in \cite{Sutherland_2020, HS_2016}
 as a subroutine to compute the $L$-polynomial mod $p$. 

In Table \ref{Table-exp1} and Table \ref{Table-exp2}, we compare our timings with an implementation of the algorithm presented in \cite{ABCMT_2019},
which computes zeta functions of cyclic covers (a more general setting) and has a time complexity of $O(p^{1/2+o(1)})$.
It is mainly implemented in Sage \cite{sage} with a key recurrence step 
implemented in C++.
\begin{remark}
  The algorithm in \cite{ABCMT_2019} is also capable of computing only the $L$-polynomial mod $p$,
but the $O(p^{1/2+o(1)})$ algorithm in \cite{Sutherland_2020} for computing $L_p(T)\bmod p$ at a single prime is more efficient.
\end{remark}

Table \ref{Table-exp1} displays the average total time to compute the $L$-polynomial 
for a single curve at all good odd primes
up to $2^{n}$. In this experiment, we 
fix $10$ genus $2$ curves with coefficients in $\Z$. For each $n$, we do the following: for each of the $10$ curves, we measure the CPU time to perform all the $L$-polynomial computations for 
odd primes of good reduction up to $2^n$, 
remove the $2$ fastest entries and 
the $2$ slowest entries, and then average the total time among the $6$ remaining entries. For the experiment using our implementation in MAGMA, we parallelized the 
computation using ten cores. In particular, for computing $L_p(T)\bmod p$ for all primes of 
good reduction up to $2^n$, we used the average polynomial-time 
algorithm from \cite{Sutherland_2020}. For the experiment using the algorithm from \cite{ABCMT_2019}, in addition to 
using parallel processing, we estimated the total runtime by sampling every $k$th prime. The value $k$ 
is chosen so that a single instance (running all good odd primes up to $2^n$ for some $n$) takes at most $1$ hour to complete. 
Our algorithm achieves a significant speed improvement. When computing all 
primes up to 
$28$-bits, we noticed a speed improvement by a factor of $\BigFactor$.

Table \ref{Table-exp2} displays the average computation time to obtain the $L$-polynomial 
for a single curve over a single odd $n$-bit good prime $p$. To obtain this result for our algorithm in MAGMA,
we select $10$ random genus~$2$ curves with coefficients in $\Z$ and compute
the $L$-polynomial with respect to $10$ random $n$-bit primes of good reduction and take the average. 
For computing $L_p(T)\bmod p$ for a single prime $p$, we used the $O(p^{1/2+o(1)})$ algorithm from \cite{Sutherland_2020}.
We do the same for the implementation of the algorithm from \cite{ABCMT_2019}.
We then remove the fastest $20\%$ and the slowest $20\%$ and compute the average of the times that remain.
Compared to \cite{ABCMT_2019}, our work achieves a speedup factor of at least $\SmallFactor$, and 
this speedup persists when extending the experiment to $40$-bit primes.

\begin{table}[!ht]
    \centering
    $\begin{array}{|r|r|r|r|r|}
        \hline
        n & \text{Compute all }L_p(T)& \text{Lift all }L_p(T) & \text{Compute all }  L_p(T) & \text{Compute all }   \\
          &   \bmod \ p\text{ (\cite{Sutherland_2020})} & \bmod \ p\text{ (our work)}   &  \text{(\cite{Sutherland_2020}} + \text{our work)} & L_p(T)\text{ (\cite{ABCMT_2019})}\\ 
          & O(\log^{4+o(1)} p) \t{ per prime}& O(\log^{2+o(1)}p) \t{ per prime}& & O(p^{1/2+o(1)}) \t{ per prime} \\ \hline
          12 & 0.052 & 0.093 & 0.145&  86.5\phantom{00} \\
          13 & 0.063 & 0.163 & 0.227&  172\phantom{.000} \\
          14 & 0.130 & 0.307 & 0.437&  324\phantom{.000} \\
          15 & 0.375 & 0.545 & 0.92\phantom{0}&  639\phantom{.000} \\
          16 & 1.26\phantom{0} & 1.01\phantom{0} & 2.27\phantom{0}&  1360\phantom{.000} \\
          17 & 3.62\phantom{0} & 1.95\phantom{0} & 5.57\phantom{0} &  2960\phantom{.000} \\
          18 & 8.48\phantom{0} & 3.60\phantom{0} & 12.1\phantom{00}&  6900\phantom{.000} \\
          19 & 19.9\phantom{00} & 6.97\phantom{0} & 26.8\phantom{00}&  16100\phantom{.000} \\
          20 & 46.3\phantom{00} & 13.9\phantom{00} & 60.2\phantom{00}&  60600\phantom{.000} \\
          21 & 110\phantom{.000} & 26.4\phantom{00} & 136\phantom{.000}&  186000\phantom{.000} \\
          22 & 264\phantom{.000} & 52.9\phantom{00} & 316\phantom{.000}&  545000\phantom{.000} \\
          23 & 625\phantom{.000} & 106\phantom{.000} & 731\phantom{.000}&  1490000\phantom{.000} \\
          24 & 1450\phantom{.000} & 205\phantom{.000} & 1660\phantom{.000} &  4110000\phantom{.000} \\
          25 & 3390\phantom{.000} & 422\phantom{.000} & 3810\phantom{.000} &  11100000\phantom{.000} \\
          26 & 7550\phantom{.000} & 952\phantom{.000} & 8500\phantom{.000} &  30900000\phantom{.000} \\
          27 & 17100\phantom{.000} & 2320\phantom{.000} & 19400\phantom{.000} &  82900000\phantom{.000} \\
          28 & 40400\phantom{.000} & 5710\phantom{.000} & 46100\phantom{.000} &  232000000\phantom{.000} \\
          \hline
    \end{array}$
\caption{Average computation time (in 2.000 GHz AMD EPYC 7713 core-seconds) for running Algorithm
\ref{alg:lift}
and \cite{ABCMT_2019} to compute $L$-polynomials of a fixed curve
for all good odd primes up to $2^{n}$.
}\label{Table-exp1}
\end{table}

\begin{table}[!ht]
    \centering 
    $\begin{array}{|r|r|r|r|r|}
        \hline
        n & \text{Compute }L_p(T)  & \text{Lift } L_p(T)  \bmod \ p  &  \text{Compute } L_p(T) & \text{Compute } L_p(T)   \\
          & \bmod \ p \text{ (\cite{Sutherland_2020})} &\text{ (our work)}   &\text{(\cite{Sutherland_2020}} + \text{our work)}& \text{(\cite{ABCMT_2019})} \\
          & O(p^{1/2+o(1)})& O(\log^{2+o(1)}p) &  & O(p^{1/2+o(1)}) \\ \hline
          12 & 0.00780 & 0.000332 & 0.00813 & 0.0745 \\ 
          13 & 0.00780 & 0.000249 & 0.00805 & 0.0804 \\ 
          14 & 0.00768 & 0.000373 & 0.00805 & 0.0851 \\ 
          15 & 0.00979 & 0.000332 & 0.0101\phantom{0} & 0.0989 \\ 
          16 & 0.0116\phantom{0} & 0.000373 & 0.0120\phantom{0} & 0.116\phantom{0} \\ 
          17 & 0.0128\phantom{0} & 0.000373 & 0.0132\phantom{0} & 0.149\phantom{0} \\ 
          18 & 0.0161\phantom{0} & 0.000332 & 0.0164\phantom{0} & 0.179\phantom{0} \\ 
          19 & 0.0191\phantom{0} & 0.000415 & 0.0195\phantom{0} & 0.226\phantom{0} \\ 
          20 & 0.0212\phantom{0} & 0.000249 & 0.0214\phantom{0} & 0.574\phantom{0} \\ 
          21 & 0.0259\phantom{0} & 0.000166 & 0.0261\phantom{0} & 0.926\phantom{0} \\ 
          22 & 0.0346\phantom{0} & 0.000207 & 0.0348\phantom{0} & 1.27\phantom{00}  \\ 
          23 & 0.0544\phantom{0} & 0.000166 & 0.0546\phantom{0} & 1.84\phantom{00}  \\ 
          24 & 0.0811\phantom{0} & 0.000166 & 0.0812\phantom{0} & 2.47\phantom{00}  \\ 
          25 & 0.133\phantom{00} & 0.000290 & 0.134\phantom{00} & 3.62\phantom{00}  \\ 
          26 & 0.207\phantom{00} & 0.000498 & 0.208\phantom{00} & 5.38\phantom{00}  \\ 
          27 & 0.362\phantom{00} & 0.000415 & 0.362\phantom{00} & 7.46\phantom{00}  \\ 
          28 & 0.562\phantom{00} & 0.000456 & 0.562\phantom{00} & 10.5\phantom{000}  \\ 
        \hline 
    \end{array}$
    \caption{Average computation time (in 2.000 GHz AMD EPYC 7713 core-seconds) for running Algorithm \ref{alg:lift}
    and \cite{ABCMT_2019} to compute $L$-polynomial for a 
    random genus $2$ curve with a random prime, $p\in [2^n, 2^{n+1}]$.}\label{Table-exp2}
\end{table}

\newpage
    \printbibliography[
        heading=bibintoc
    ]

\end{document}